\PassOptionsToPackage{hyphens}{url}
\documentclass[a4paper,USenglish,cleveref,autoref,thm-restate,nolineno]{socg-lipics-v2021}

\pdfoutput=1 
\hideLIPIcs  

\title{On the Twin-Width of Smooth Manifolds} 

\author{\'Edouard Bonnet}{Univ Lyon, CNRS, ENS de Lyon, Université Claude Bernard Lyon 1, LIP UMR5668, France \and \url{http://perso.ens-lyon.fr/edouard.bonnet/}}{edouard.bonnet@ens-lyon.fr}{https://orcid.org/0000-0002-1653-5822}{}

\author{Kristóf Huszár}{Institute of Geometry, Graz University of Technology, Austria \and \url{https://kristofhuszar.github.io/}}{kristof.huszar@tugraz.at}{https://orcid.org/0000-0002-5445-5057}{}
\authorrunning{\'E. Bonnet and K. Huszár} 

\Copyright{\'Edouard Bonnet and Kristóf Huszár} 



\category{} 

\relatedversion{} 

\funding{The authors have been supported by the French National Research Agency through the project TWIN-WIDTH with reference number ANR-21-CE48-0014.}

\nolinenumbers 

\EventEditors{John Q. Open and Joan R. Access}
\EventNoEds{2}
\EventLongTitle{42nd Conference on Very Important Topics (CVIT 2016)}
\EventShortTitle{CVIT 2016}
\EventAcronym{CVIT}
\EventYear{2016}
\EventDate{December 24--27, 2016}
\EventLocation{Little Whinging, United Kingdom}
\EventLogo{}
\SeriesVolume{42}
\ArticleNo{23}

\usepackage{mathtools}
\usepackage{overpic}

\newcommand{\altaltmanifold}{\mathscr{Q}}
\newcommand{\altmanifold}{\mathscr{N}}
\newcommand{\atlas}{\mathcal{A}}
\newcommand{\ball}{\mathscr{B}}
\newcommand{\bigsimplex}{\Sigma}
\newcommand{\clstar}{\overline{\operatorname{st}}}
\newcommand{\class}{\mathscr{G}}
\newcommand{\coloring}{\mathscr{C}}
\newcommand{\complex}{\mathscr{X}}
\newcommand{\diagram}{\mathscr{D}}
\newcommand{\dual}{\Gamma}
\newcommand{\embedding}{\mathscr{E}}

\newcommand{\griddiag}{D}
\newcommand{\hcubic}{\mathbf{H}}
\newcommand{\image}{\operatorname{im}}
\newcommand{\integer}{m}
\newcommand{\interior}{\operatorname{int}}
\newcommand{\knot}{\mathscr{K}}

\newcommand{\manifold}{\mathscr{M}}

\newcommand{\polydecomp}{\mathscr{R}}
\newcommand{\prism}{P}
\newcommand{\prisms}{\mathscr{P}}

\newcommand{\red}{\operatorname{red}}
\newcommand{\seq}{\mathfrak{S}}
\newcommand{\set}{\mathscr{S}}

\newcommand{\sd}{\operatorname{subd}}
\newcommand{\surface}{S}
\newcommand{\tri}{\mathscr{T}}

\newcommand{\tww}{\operatorname{tww}}
\newcommand{\vertices}{\mathscr{V}}
\newcommand{\width}{\operatorname{w}}

\newcommand{\geomrel}[1]{\|{#1}\|}

\newcommand{\mynum}[1]{{\color{lipicsGray}\sffamily\bfseries{#1}}}

\renewcommand{\leq}{\leqslant}
\renewcommand{\geq}{\geqslant}

\newcommand\smallscale{0.884615384615} 

\begin{document}

\maketitle

\begin{abstract}
Building on Whitney's classical method of triangulating smooth manifolds, we show that every compact $d$-dimensional smooth manifold admits a triangulation with dual graph of twin-width at~most~$d^{O(d)}$. In particular, it follows that every compact \mbox{$3$-manifold} has a triangulation with dual graph of bounded twin-width. This is in sharp contrast to the case of treewidth, where for any natural number $n$ there exists a closed 3-manifold such that every triangulation thereof has dual graph with treewidth at least $n$.
To establish this result, we bound the twin-width of the incidence graph of the $d$-skeleton of the second barycentric subdivision of the $2d$-dimensional hypercubic honeycomb. We also show that every compact, piecewise-linear (hence smooth) $d$-dimensional manifold has triangulations where the dual graph has an arbitrarily large twin-width.
\end{abstract}

\section{Introduction}
\label{sec:intro}

Structural graph parameters have become increasingly important in computational topology in the past two decades. This is mainly due to the emergence of fixed-parameter tractable (FPT) algorithms for problems on knots, links \cite{burton2018homfly, makowsky2005coloured, makowsky2003parameterized, maria2021parametrized} and 3-manifolds \cite{burton2017courcelle, burton2016parameterized, burton2018algorithms, pettersson2014fixed, burton2013complexity},\footnote{Also see \cite{bagchi2016efficient} for an FPT algorithm checking tightness of (weak) pseudomanifolds in arbitrary dimensions.} most of which are known to be NP-hard in general.
Although these FPT algorithms may have exponential worst-case running time, on inputs with bounded \emph{treewidth} they are guaranteed to terminate in polynomial (or even in linear) time.\footnote{Here the term \emph{input} refers either to a link diagram $\diagram$, or to a 3-manifold triangulation $\tri$. In the first case the treewidth means the treewidth of $\diagram$ considered as a $4$-regular graph, in the second case it means the treewidth of the dual graph $\dual(\tri)$ of $\tri$. The running times are measured in terms of the \emph{size} of the input. The size of a link diagram is defined as the number of its crossings, and the size of a 3-manifold triangulation is the number of its tetrahedra. More definitions are given in \Cref{sec:prelims}.} In addition, some of these algorithms have been implemented in software packages such as \texttt{Regina}, providing practical tools for researchers in topology \cite{burton2013regina,regina}.\footnote{See \cite{barnatan2007fast} for an implemented algorithm to effectively compute certain Khovanov homology groups of knots. This algorithm is conjectured to be FPT in the \emph{cutwidth} of the input knot diagram, cf.\ \cite[Section 6]{barnatan2007fast}.}

The success of the above algorithms naturally leads to the following question.
Given a~3-manifold $\manifold$ (resp.\ knot~$\knot$), what is the smallest treewidth that the dual graph of a~triangulation of $\manifold$ (resp.\ a~diagram of $\knot$) may have?\footnote{For links and knots, this question was respectively asked in \cite[Section 4]{makowsky2003parameterized} and \cite[p.\ 2694]{MFOReports}.}
Motivated by this challenge, in recent years several results have been obtained that reveal quantitative connections between \emph{topological invariants} of knots and 3-manifolds, and \emph{width parameters} associated with their diagrams \cite{mesmay2019treewidth, lunel2023structural} and triangulations \cite{huszar2020combinatorial, huszar2022pathwidth, huszar2019manifold, huszar2023width, huszar2019treewidth, maria2019treewidth}, respectively.
It turns out that topological properties of 3-manifolds may prohibit the existence of ``thin'' triangulations.\footnote{In particular, for non-Haken 3-manifolds of large Heegaard genus \cite{huszar2019treewidth} or Haken 3-manifolds with a~``complicated'' JSJ decomposition \cite{huszar2023width}, the dual graph of \emph{any} triangulation must also have large treewidth.}\textsuperscript{,}\footnote{For results where the treewidth of a knot diagram is bounded below by topological properties of the underlying knot, see  \cite{mesmay2019treewidth, lunel2023structural}.} At the same time, geometric or topological descriptions of 3-manifolds can also give strong hints on how to triangulate them so that their dual graphs have constant pathwidth or treewidth, or at least bounded in terms of a topological invariant of these 3-manifolds.\footnote{This is case with Seifert fibered spaces~\cite{huszar2019manifold} or hyperbolic 3-manifolds \cite{huszar2022pathwidth, maria2019treewidth}.}

In this work we establish similar results for another graph parameter called \emph{twin-width}.\footnote{We will denote by $\tww(G)$ the twin-width of a~graph~$G$.} Introduced in \cite{bonnet2022twin-width-1}, this notion has been subject of growing interest and found many algorithmic applications in recent years.\footnote{For an introduction to twin-width and an overview of its applications, see \cite{bonnet2024habil} and the references therein.}
Namely, on classes of effectively\footnote{A~class has \emph{effectively bounded twin-width} if it has bounded twin-width, and contraction sequences of width $O(1)$ (objects witnessing the twin-width upper bound) can be found in polynomial time; see~\cref{ssec:tww} for the definitions of \emph{contraction sequences} and \emph{twin-width}.} bounded twin-width, first-order properties can be decided efficiently\footnote{More precisely, there is a~\emph{fixed-parameter tractable} algorithm that, given a~first-order sentence $\varphi$ and an $n$-vertex graph $G$ with a~contraction sequence of width~$d$, decides if $G$ satisfies $\varphi$ in time $f(\varphi,d) \cdot n$, for some computable function~$f$.}~\cite{bonnet2022twin-width-1} (also see~\cite{bonnet2021twin-width-3} for improved running times on specific problems definable in first-order logic), first-order queries can be enumerated fast~\cite{gajarsky2022twin-width}, and enhanced approximation algorithms can be designed for several graph optimization problems~\cite{berge2023approximating}.
Besides, classes of bounded twin-width are fairly general and broad.
They for instance include classes of bounded tree-width, and even its dense analogue, clique-width, classes excluding a~minor, proper permutation classes, $d$-dimensional grid graphs~\cite{bonnet2022twin-width-1}, some classes of cubic expanders~\cite{bonnet2022twin-width-2}, segment intersection graphs without biclique subgraphs of a~fixed size~\cite{bonnet2022twin-width8}, and modifications definable in first-order logic (called first-order \emph{transductions}) of all these classes~\cite{bonnet2022twin-width-1}. 
We observe that the definition of twin-width can readily be lifted to binary structures (i.e., edge-colored multigraphs).

Our first result shows that a compact $d$-dimensional smooth manifold always has a~triangulation with dual graph of twin-width bounded in terms of~$d$.

\begin{theorem}
\label{thm:tww-mfd}
Any compact \mbox{$d$-dimensional} smooth manifold admits a triangulation with dual graph of twin-width at~most~$d^{O(d)}$.
\end{theorem}
Since every 3-manifold is smooth \cite{moise1952affine} (also see \cite{ manolescu2016lectures}), the following corollary is immediate.
We recall that $\tww(G)$ denotes the twin-width of a~graph~$G$, and $\dual(\tri)$ denotes the dual graph of a~triangulation $\tri$.

\begin{corollary}
\label{cor:tww-3mfd}
There exists a universal constant $C > 0$ such that every compact \mbox{$3$-dimensional} manifold $\manifold$ admits a triangulation $\tri$ with $\tww(\dual(\tri)) \leq C$.
\end{corollary}

This is in sharp contrast to the case of treewidth, for which it is known that for every $n \in \mathbb{N}$ there are infinitely many 3-manifolds where the smallest treewidth of the dual graph of \emph{every} triangulation is at least $n$ \cite{huszar2023width, huszar2019treewidth}.
Complementing \Cref{thm:tww-mfd}, we also show that for any fixed $d \geq 3$, the $d$-dimensional triangulations of large twin-width are abundant (\Cref{thm:trg-not-small}). Moreover we show that any piecewise-linear (hence smooth) manifold of dimension at least three admits triangulations with dual graph of arbitrarily large twin-width.

\begin{theorem}
\label{thm:tww-large}
Let $d \geq 3$ be an integer. For every compact $d$-dimensional piecewise-linear manifold $\manifold$ and natural number $n \in \mathbb{N}$, there is a triangulation $\tri$ of $\manifold$ with $\tww(\dual(\tri)) \geq n$.
\end{theorem}

\begin{remark}
\label{rem:graphs-on-surfaces}
The assumption of $d \geq 3$ in \Cref{thm:tww-large} is essential. Indeed, the dual graph of any triangulation of the genus-$g$ surface $\surface_g$ is, in particular, a graph that embeds into $\surface_g$ and such graphs are known to have twin-width bounded above by $c(\sqrt{g}+1)$ for some universal constant $c > 0$ \cite{kral2023tww-surfaces}.\footnote{This bound is sharp up to a constant multiplicative factor \cite{kral2023tww-surfaces}. For $g=0$, we know that planar graphs have twin-width at most eight \cite{hlineny2023tww-planar}, and there are planar graphs with twin-width equal to seven \cite{kral2023planar}.}
\end{remark}

\subparagraph*{Outline of the paper.} In \Cref{sec:prelims} we review the relevant notions from graph theory and topology. In \Cref{sec:whitney} we recall Whitney's seminal work on triangulating smooth manifolds. This is followed by a detailed proof of \Cref{thm:tww-mfd} in \Cref{sec:proof}. Finally, in \Cref{sec:large} we prove the complementary results about triangulations with dual graph of large twin-width.

\section{Preliminaries}
\label{sec:prelims}

\subparagraph*{Basic notation.} For a finite set $S$ we let $|S|$ denote its cardinality, while for a real number $x$ we let $|x|$ denote its absolute value. For a positive integer $k \leq |S|$, we let $\binom{S}{k}$ denote the set of $k$-element subsets of $S$. For a positive integer $n$, we let $[n]$ denote the set of all positive integers up to $n$, and for two real numbers $a \leq b$, we use $[a,b]$ to denote the closed interval $\{x \in \mathbb{R} : a \leq x \leq b\}$. For a vector $y = (y_1,\ldots,y_d) \in \mathbb{R}^d$ we use $\|y\|$ to denote its Euclidean norm, i.e, $\|y\|^2 = \sum_{i = 1}^d y_i^2$. However, if $\complex$ is a (cubical or simplicial) complex, then $\geomrel{\complex}$ refers to its geometric realization.

\subsection{Trigraphs, contraction sequences and twin-width}
\label{ssec:tww}

Following \cite[Sections 3 and 4]{bonnet2022twin-width-1}, in this section we review the graph-theoretic notions central to our work and collect some basic, yet important facts about twin-width.

\subparagraph*{Trigraphs.} A \emph{trigraph} $G$ is a triple $G = (V,E,R)$, where $V$ is a finite set of \emph{vertices}, and $E,R \subseteq \binom{V}{2}$ are two disjoint subsets of pairs of vertices called \emph{black edges} and \emph{red edges}, respectively. We also refer to the sets of vertices, black edges and red edges of a given trigraph $G$ as $V(G)$, $E(G)$ and $R(G)$, respectively. Any simple graph $G=(V,E)$ may be regarded as a trigraph $(V,E,R)$ with $R = \emptyset$. For a vertex $v \in V(G)$ the \emph{degree $\deg(v)$ of $v$} is the number of edges incident to it, i.e., $\deg(v) = |\{e \in E \cup R : v \in e\}|$.  Additionally, the \emph{red degree $\deg_R(v)$ of $v$} is the number of red edges incident to it, i.e., $\deg_R(v) = |\{e \in R : v \in e\}|$.
A~trigraph $G$ for which $\deg_R(G) = \max_{v \in V(G)}{\deg_R(v)} \leq b$ is called a \emph{$b$-trigraph}. Given two trigraphs $G=(V,E,R)$ and $G' = (V',E',R')$, we say that $G'$ is a \emph{subtrigraph} of $G$, if $V' \subseteq V$, $E' \subseteq E \cap \binom{V'}{2}$ and $R' \subseteq R \cap \binom{V'}{2}$.\footnote{As usual, subtrigraphs of graphs (those without any red edges) will also be called subgraphs.} In addition, if $E' = E \cap \binom{V'}{2}$ and $R' = R \cap \binom{V'}{2}$, then we say that $G'$ is an \emph{induced subtrigraph} of $G$. For a trigraph $G$ and a subset $S \subseteq V(G)$ of its vertices, $G - S$ denotes the induced subtrigraph of $G$ with vertex set $V(G) \setminus S$.

\subparagraph*{Contraction sequences and twin-width.} Let $G = (V,E,R)$ be a trigraph and $u,v \in V$ be two arbitrary distinct vertices of $G$. We say that the trigraph $G / u,v = (V',E',R')$ is obtained from $G$ by \emph{contracting} $u$ and $v$ into a new vertex $w$ if \mynum{1.}~$V' = (V \setminus \{u,v\}) \cup \{w\}$, \mynum{2.}~$G - \{u,v\} = (G/u,v) - \{w\}$ and \mynum{3.}~for any $x \in V'\setminus\{w\} = V\setminus\{u,v\}$ we have
\begin{itemize}
	\item $\{w,x\} \in E'$ if and only if $\{u,x\} \in E$ and $\{v,x\} \in E$,
	\item $\{w,x\} \notin E' \cup R'$ if and only if $\{u,x\} \notin E \cup R$ and $\{v,x\} \notin E \cup R$, and
	\item $\{w,x\} \in R'$ otherwise.
\end{itemize}

We call the trigraph $G / u,v$ a \emph{contraction} of $G$.
A sequence $\seq = (G_1,\ldots,G_m)$ of trigraphs is a \emph{contraction sequence} if $G_{i+1}$ is a contraction of $G_i$ for every $1 \leq i \leq m-1$. Note that $|V(G_{i+1})| = |V(G_i)|-1$. We use the notation ``$\seq\colon G_1 \leadsto G_m$'' to indicate that the trigraphs $G_1$ and $G_m$ are initial and terminal entries of the contraction sequence $\seq$.
The \emph{width} $\width(\seq)$ of a contraction sequence $\seq = (G_1,\ldots,G_m)$ is defined as $\width(\seq) = \max_{1 \leq i \leq m}\deg_R(G_i)$, i.e., the largest red degree of any vertex of any trigraph in $\seq$.
Now, the \emph{twin-width} $\tww(G)$ of a trigraph $G$ is defined as the smallest width of any contraction sequence $(G_1,\ldots,G_{|V(G)|})$ with $G_1 = G$ and $G_{|V(G)|}=\bullet$, where $\bullet$ denotes the trigraph consisting of a single vertex.

\paragraph*{Some properties of twin-width; grid graphs}

We conclude this section by collecting some properties of twin-width that we will rely on later. The first one states that twin-width is monotonic under taking induced subtrigraphs and is a simple consequence of the definitions (cf.\ \cite[Section 4.1]{bonnet2022twin-width-1}).

\begin{proposition}
\label{prop:induced}
If $H$ is an induced subtrigraph of a trigraph $G$, then $\tww(H) \leq \tww(G)$.
\end{proposition}

\subparagraph*{Smallness.} An infinite class $\class$ of graphs is \emph{small} if there exists a constant $c > 1$ such that for every $n \in \mathbb{N}$ the class $\class$ contains at most $n!c^n$ labeled graphs on $n$ vertices.
The next theorem says that every graph class of bounded twin-width---i.e., for which there exists a constant $C > 0$, such that $\tww(G) \leq C$ for every graph $G$ in the class---is small.

\begin{theorem}[{\cite[Theorem~2.5]{bonnet2022twin-width-2}}]
\label{thm:small}
Every graph class with bounded twin-width is small.
\end{theorem}

Now let $s$ be a non-negative integer.
The \emph{$s$-subdivision} of $G$ is the graph $\sd_s(G)$ obtained from $G$ by subdividing each edge in $E(G)$ exactly $s$ times.
A simple counting argument together with \Cref{thm:small} yields the following:

\begin{proposition}
\label{prop:k-reg-s-sub}
For any fixed integers $k \geq 4$ and $s \geq 0$, the class $\sd_s(\class_{k})$ of $s$-subdivisions of $k$-regular\footnote{A simple graph $G=(V,E)$ is \emph{$k$-regular} if every vertex $v \in V$ has degree $\deg(v)=k$.} simple graphs is not small, hence has unbounded twin-width.
\end{proposition}

This proposition follows from the adaptation of an argument given in the first paragraph of \cite[Section 3]{dvorak2010small}. For completeness, we spell out this proof below.

\begin{proof}[Proof of \Cref{prop:k-reg-s-sub}] Let $N_k(m)$ be the number of labeled $k$-regular simple graphs on $m$~vertices. Note that if $N_k(m) > 0$, then $km$ is even; which we now assume. It is known (cf.\ \cite[Section\ 6.4.1]{noy2015graphs}) that, asymptotically
\begin{align}
	N_k(m) \sim \exp\left(\frac{1-k^2}{4}\right) \frac{(km)!}{(km/2)! \cdot 2^{km/2} \cdot (k!)^m} = \Omega\left(\frac{(km/2)!}{(k!)^m}\right).
	\label{eq:k-reg}
\end{align}
Further, let $N_k^{(s)}(n)$ be the number of $n$-vertex graphs in the class $\sd_s(\class_{k})$ of $s$-subdivisions of $k$-regular simple graphs. If a graph $G \in \sd_s(\class_{k})$ has $n$ vertices, then $n = m + \frac{km}{2}s$, where $m$ is the number of vertices of $G$ of degree $k$. Note that such a graph $G$ can be obtained by first choosing a $k$-regular labeled graph $H$ on $m$ vertices, then ordering the remaining $n-m = kms/2$ vertices arbitrarily and evenly distributing them to the edges of $H$ according to some fixed ordering of $E(H)$. It follows that
\begin{align}
	N_k^{(s)}(n) \geq \binom{n}{m}\cdot N_k(m) \cdot (n-m)! = n! \frac{N_k(m)}{m!} \overset{\eqref{eq:k-reg}}{=}  n! \cdot \Omega\left(\frac{(km/2)!}{m!\cdot(k!)^m}\right).
	\label{eq:s-k-reg}
\end{align}
Since $k \geq 4$, we have $km/2 \geq 2m$. This, together with $(2m)! \geq 2^m(m!)^2$ and \eqref{eq:s-k-reg} implies
\begin{align}
	N_k^{(s)}(n) \geq n! \cdot \Omega\left(\frac{m!}{(2 \cdot k!)^m}\right).
\end{align}
Recall that $m = 2n/(2+ks)$. In particular, $m$ is proportional to $n$. Hence $m! / (2 \cdot k!)^m$ grows faster than $c^n$ for any fixed constant $c > 1$. Thus the graph class $\sd_s(\class_{k})$ is not small.
\end{proof}

\subparagraph*{Grid graphs.} The \emph{$d$-dimensional $n$-grid} $P_n^d$ is the graph with vertex set $V(P_n^d) = [n]^d$, and $\{u,v\} \in E(P_n^d)$ for two vertices $u=(u_1, \ldots, u_d)$ and $v = (v_1,\ldots,v_d)$ if and only if $\sum_{i=1}^d |u_i - v_i| = 1$.
The next result states that $\tww(P_n^d) \leq 3d$ irrespective of the value of $n$.

\begin{theorem}[Theorem 4 in \cite{bonnet2022twin-width-1}]
\label{thm:tww-grid}
For every positive $d$ and $n$, the $d$-dimensional $n$-grid $P_n^d$ has twin-width at most $3d$.
\end{theorem}

The \emph{$d$-dimensional $n$-grid with diagonals} $\griddiag_{n,d}$ is the graph with $V(\griddiag_{n,d})=[n]^d$, and $\{u,v\} \in E(\griddiag_{n,d})$ for two vertices $u=(u_1, \ldots, u_d)$ and $v = (v_1,\ldots,v_d)$ if and only if $\max_{i=1}^d |u_i - v_i| \leq 1$. Now, for a given trigraph  $G = (V,E,R)$ we set $\red(G) = (V,\emptyset,E \cup R)$. In words, $\red(G)$ is the trigraph obtained from $G$ by replacing every black edge of $G$ by a red edge between the same vertices. With this notation $\red(\griddiag_{n,d})$ is the \emph{$d$-dimensional \textbf{red} $n$-grid with diagonals}, i.e., $\red(\griddiag_{n,d}) = ([n]^d, \emptyset, E(\griddiag_{n,d}))$. Clearly, $\tww(\griddiag_{n,d}) \leq \tww(\red(\griddiag_{n,d}))$.

\begin{theorem}[Lemma 4.4 in \cite{bonnet2022twin-width-1}]
\label{thm:tww-grid-diag}
For every positive $d$ and $n$, every subtrigraph of the $d$-dimensional red $n$-grid with diagonals $\red(\griddiag_{n,d})$ has twin-width at most $2(3^d-1)$.
\end{theorem}

\subsection{Background in topology}
\label{ssec:topology}

For general background in (combinatorial and differential) topology we refer to \cite{prasolov2006elements}.

\subsubsection{Simplicial and cubical complexes}
\label{ssec:complexes}

\subparagraph*{Abstract simplicial complexes.} Given a finite ground set $\set$, an \emph{abstract simplicial complex} (or \emph{simplicial complex}, for short) $\complex$ over $\set$ is a downward closed subset  of the power set $2^\set$, i.e., $\mathcal{F} \in \complex$ and $\mathcal{F}' \subset \mathcal{F}$ imply $\mathcal{F}' \in \complex$.
Any element of $\complex$ is called a \emph{face} or \emph{simplex} of $\complex$, and for $\sigma \in \complex$ the \emph{dimension of $\sigma$} is defined as $\dim\sigma = |\sigma| - 1$.
The \emph{dimension of $\complex$}, denoted by $\dim\complex$, is then the maximum dimension of a~face of~$\complex$.
If $\dim\complex = d$, we also say that $\complex$ is a~\emph{simplicial $d$-complex}.
For $0\leq i \leq \dim\complex$, we let $\complex(i) = \{\sigma \in \complex : \dim\sigma = i\}$ denote the set of $i$-dimensional faces (or $i$-faces, or $i$-simplices) of $\complex$.
The \emph{$i$-skeleton} $\complex_i = \bigcup_{j=0}^i \complex(j)$ of $\complex$ is the union of all faces of $\complex$ up to dimension $i$.  Note that any simple graph $G = (V,E)$ can be regarded as a $1$-dimensional simplicial complex $\complex_G$ with $\complex_G(0) = \{\{v\} : v \in V\}$ and $\complex_G(1) = E$. For $0 \leq i \leq 3$ the $i$-simplices of a simplicial complex $\complex$ are respectively called the \emph{vertices}, \emph{edges}, \emph{triangles}, and \emph{tetrahedra} of $\complex$.

\subparagraph*{Geometric realization of simplicial complexes.} Every abstract simplicial complex $\complex$ may be realized geometrically as follows. To each abstract $i$-simplex $\sigma = \{v_{0},v_{1},\ldots,v_{i}\} \in \complex$ we associate a \emph{geometric $i$-simplex} $\|\sigma\| = [v_{0},v_{1},\ldots,v_{i}] \subset \mathbb{R}^i$ defined as the convex hull of $i+1$ affinely independent points in $\mathbb{R}^i$. We equip $\|\sigma\|$ with the subspace topology inherited from~$\mathbb{R}^i$. Next, we consider the disjoint union $\bigsqcup_{\sigma \in \complex}\|\sigma\|$ of these geometric simplices, and perform identifications along their faces that reflect their relationship in $\complex$. The resulting space $\|\complex\|$, equipped with the quotient topology, is called the \emph{geometric realization of $\complex$}, see \Cref{fig:geometric-example}. The geometric realization of a simplicial complex is unique up to homeomorphism.

\begin{theorem}[folklore; cf.\ {\cite[Theorem 3.15]{prasolov2006elements}}]
	Let $\|\complex\|$ be the geometric realization of a~$d$-dimensional simplicial complex $\complex$. Then there exists an embedding (i.e.,\ a continuous, injective map) $f\colon \|\complex\| \rightarrow \mathbb{R}^{2d+1}$. Furthermore, $f$ can be chosen to be simplex-wise linear.
	\label{thm:geometric-realization}
\end{theorem}

\begin{remark}
\Cref{thm:geometric-realization} generalizes the well-known fact that every graph admits a~straight-line embedding in $\mathbb{R}^3.$
\end{remark}

\subparagraph*{Cubical complexes.} Analogous to simplicial complexes, a \emph{cubical complex} $\complex$ over a ground set $\set$ is a set system $\complex \subset 2^\set$ that consists of ``cubes'' instead of simplices. The terminology is the same as in the simplicial case. The only difference is that a \emph{geometric $i$-cube} is a~topological space homeomorphic to $[0,1]^i$, where $[0,1]$ denotes the closed unit interval.

Cubical or simplicial complexes in this paper will typically be defined geometrically, and as such they will naturally come with a geometric realization.

\begin{figure}[htbp]

\begin{subfigure}{0.4\textwidth}
    \centering
    \includegraphics{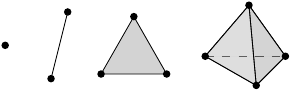}
    \bigskip
    \caption{Geometric $i$-simplices ($i=0,1,2,3$).}
    \label{fig:simplices}
\end{subfigure}%
\hfill%
\begin{subfigure}{0.5\textwidth}
    \centering
    \includegraphics{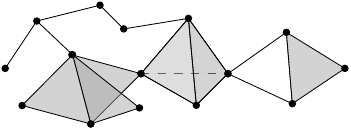}
    \smallskip
    \caption{Geometric realization of a simplicial 3-complex.}
    \label{fig:simplicial-complex}
\end{subfigure}%

\medskip

\begin{subfigure}{0.4\textwidth}
    \centering
    \includegraphics{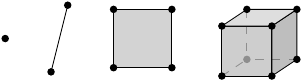}
    \bigskip
    \caption{Geometric $i$-cubes ($i=0,1,2,3$).}
    \label{fig:cubes}
\end{subfigure}%
\hfill%
\begin{subfigure}{0.5\textwidth}
    \centering
    \includegraphics{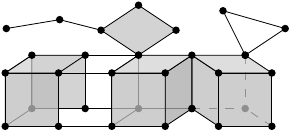}
    \smallskip
    \caption{Geometric realization of a cubical 3-complex.}
    \label{fig:cubical-complex}
\end{subfigure}%

\caption{The geometric perspective on simplicial and cubical complexes.}
\label{fig:geometric-example}
\end{figure}

\subparagraph*{The hypercubic honeycomb.} Let $n$ and $d$ be positive integers and consider the $d$-di\-men\-sional cube $[1,n]^d \subset \mathbb{R}^d$. The \emph{$d$-di\-men\-sional hypercubic honeycomb} $\hcubic^{d,n}$ is a cubical $d$-complex that decomposes $[1,n]^d$ into $(n-1)^d$ geometric cubes. The properties of this familiar object play an important role in this work, so we describe it for completeness. We define $\hcubic^{d,n}$ geometrically, in a bottom-up fashion.
First, the vertex set $\hcubic^{d,n}(0)$ consists of precisely those points in $[1,n]^d$, which have only integral coordinates. Next, a $1$-cube (i.e., an edge) is attached along its endpoints to vertices $u=(u_1, \ldots, u_d)$ and $v = (v_1,\ldots,v_d)$ in $\hcubic^{d,n}(0)$ if and only if $\sum_{i=1}^d |u_i - v_i| = 1$. Thus the 1-skeleton $\hcubic^{d,n}_1$ is just the $d$-dimensional grid graph $P^d_n$ encountered in the end of \Cref{ssec:tww}. Finally, the higher dimensional skeleta of $\hcubic^{d,n}$ are \emph{induced} by its 1-skeleton: for each subcomplex $\mathcal{Y} \subset \hcubic^{d,n}_i$ isomorphic to the boundary of a $(i+1)$-cube, we attach an $(i+1)$-cube to $\hcubic^{d,n}_i$ along $\mathcal{Y}$. In words, starting from the  1-skeleton $\hcubic^{d,n}_1$, whenever we have the possibility to attach a cube of dimension at least two (because its boundary is already present), we attach it. See, e.g., \Cref{fig:induced_complex}.

\begin{figure}[htbp]
	\centering
	
	\smallskip
	\includegraphics{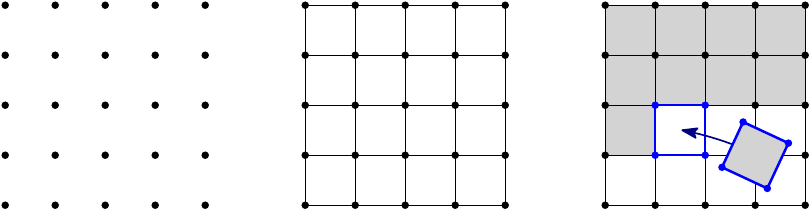}
	
	\smallskip
	
	\caption{Constructing the complex $\hcubic^{2,5}$. Its 1-skeleton $\hcubic^{2,5}_1$ is isomorphic to the grid $P_5^2$.}
	\label{fig:induced_complex}
\end{figure}

\subparagraph*{Pure complexes and their dual graphs.} A (cubical or simplicial) complex $\complex$ is \emph{pure} if every face of $\complex$ is contained in a face of dimension $\dim\complex$. It follows that, for every $i$ with $0 \leq i \leq \dim\complex$, the $i$-skeleton $\complex_i$ of a pure complex $\complex$ is also pure. Examples of pure complexes include nonempty graphs without isolated vertices, or triangulations of manifolds (see \Cref{ssec:manifolds}).
Given a pure complex $\complex$, the \emph{dual graph} $\dual(\complex_i) = (V,E)$ of its $i$-skeleton is defined as the graph, where the vertex set $V$ corresponds to the set $\complex(i)$ of $i$-faces, and $\{\sigma,\tau\} \in E$ if and only if $\sigma$ and $\tau$ share an $(i-1)$-dimensional face in $\complex$, see \Cref{fig:trg-fan-dual}. 

\begin{figure}[ht]
\centering
\begin{subfigure}{0.4\textwidth}
    \centering
    \includegraphics[scale=\smallscale]{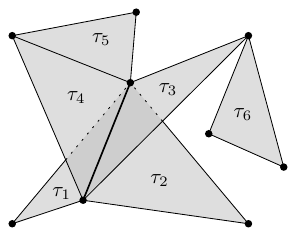}
    \caption{A pure simplicial 2-complex $\complex$ formed by six triangles $\tau_1$, $\tau_2$, $\tau_3$, $\tau_4$, $\tau_5$, and $\tau_6$, four of which meet along a single edge.}
    \label{fig:triangle-fan}
\end{subfigure}%
\hfill%
\begin{subfigure}{0.4\textwidth}
    \centering
    \includegraphics[scale=\smallscale]{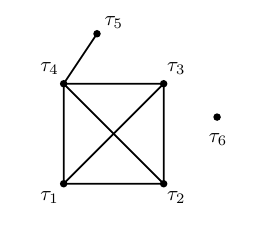}
    \caption{The dual graph $\dual(\complex_2)$ of $\complex_2$. As $\tau_6$ shares no edge with any other triangle in $\complex$, its corresponding vertex is isolated.}
    \label{fig:triangle-fan-dual}
\end{subfigure}%
\caption{Example of a pure simplicial $2$-complex and its dual graph.}
\label{fig:trg-fan-dual}
\end{figure}

\subsubsection{Barycentric subdivisions}
\label{ssec:barycentric}

Given a (cubical or simplicial) complex $\complex$, its \emph{barycentric subdivision} is a \textbf{simplicial} complex $\complex'$ defined abstractly as follows. For the ground set $\set'$ of $\complex'$ we have $\set' = \complex$. A~$(k+1)$-tuple $\{\sigma_0,\ldots,\sigma_k\} \subset \set'$ forms a $k$-simplex of $\complex'$ if and only if $\sigma_i \subset \sigma_j$ for every $0 \leq i < j \leq k$. We denote the 2\textsuperscript{nd} (resp.\ $\ell$\textsuperscript{th}) iterated barycentric subdivision of a complex $\complex$ by $\complex''$ (resp.\ $\complex^{(\ell)}$). See \Cref{fig:barycentric,fig:subdivision} for examples. The following are simple consequences of the definitions.

\begin{observation}
For any $0 \leq l \leq k$, the $k$-simplex has $\binom{k}{l}$ $l$-faces.
\label{claim:simplex-faces}
\end{observation}

\begin{observation}
The barycentric subdivision of the $k$-simplex contains $k!$ $k$-simplices.
\label{claim:bary-simplex}
\end{observation}

\begin{observation}
The barycentric subdivision of the $k$-cube contains $2^k k!$ $k$-simplices.
\label{claim:bary-cube}
\end{observation}

\begin{figure}[ht]
\centering
\begin{subfigure}{1\textwidth}
	\includegraphics{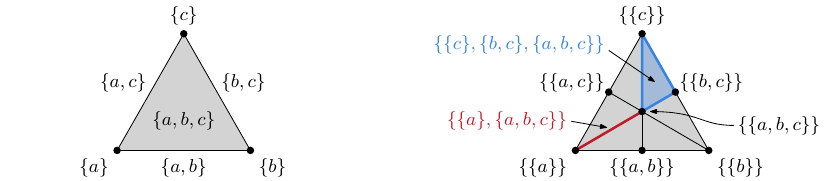}
	\caption{A triangle presented as a simplicial complex (left) and its barycentric subdivision (right).}
	\label{fig:barycentric-trg}
\end{subfigure}

\bigskip

\begin{subfigure}{1\textwidth}
	\includegraphics{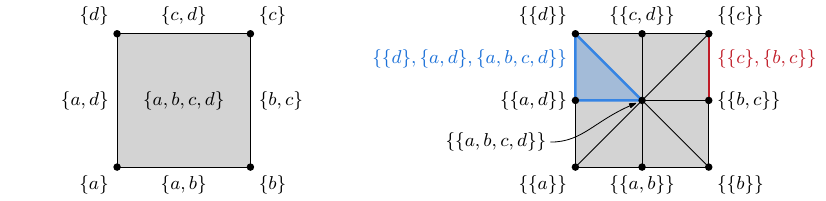}
	\caption{A square presented as a cubical complex (left) and its barycentric subdivision (right).}
	\label{fig:barycentric-square}
\end{subfigure}

	\caption{The effect of barycentric subdivision on a triangle and on a square.}
	\label{fig:barycentric}
\end{figure}

\begin{figure}[ht]
\centering
\begin{subfigure}{100pt}
    \includegraphics{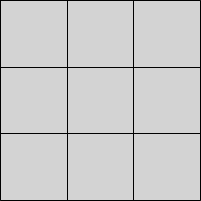}
    \caption{$\hcubic^{2,4}$}
    \label{fig:grid}
\end{subfigure}%
\hfill%
\begin{subfigure}{100pt}
    \includegraphics{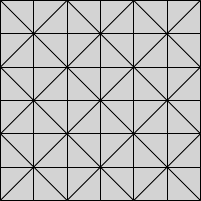}
    \caption{$(\hcubic^{2,4})'$}
    \label{fig:sd1}
\end{subfigure}%
\hfill%
\begin{subfigure}{100pt}
    \includegraphics{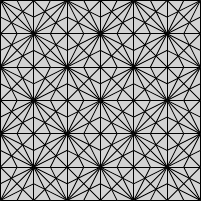}
    \caption{$(\hcubic^{2,4})''$}
    \label{fig:sd2}
\end{subfigure}
        
\caption{The effect of two barycentric subdivisions on the 2-dimensional hypercubic honeycomb.}
\label{fig:subdivision}
\end{figure}

\subsubsection{Manifolds and their triangulations}
\label{ssec:manifolds}

Manifolds---the main objects of interest in this paper---can be regarded as higher dimensional analogs of surfaces. A \emph{$d$-dimensional topological manifold with boundary} (or \emph{$d$-manifold}, for short) is a topological space\footnote{As a topological space a manifold is required to be \emph{second countable} \cite[p.\ 2]{prasolov2006elements} and \emph{Hausdorff} \cite[p.\ 87]{prasolov2006elements}.} $\manifold$, where every point has an open neighborhood homeomorphic to $\mathbb{R}^d$, or to the closed upper half-space $\{(x_1,\ldots,x_d) \in \mathbb{R}^d : x_1 \geq 0\}$. The points of $\manifold$ that do \textbf{not} have a neighborhood homeomorphic to $\mathbb{R}^d$ constitute the \emph{boundary} $\partial\manifold$ of $\manifold$. If a manifold $\manifold$ satisfies $\partial\manifold = \emptyset$, then $\manifold$ is called a \emph{closed} manifold.

In this paper $\manifold$ always denotes a compact manifold.

\subparagraph*{Smooth manifolds.} The main result of this paper (\Cref{thm:tww-mfd}) applies for manifolds that have an additional property, namely \emph{smoothness}. As we will not need to work with the actual definition of smoothness, but merely rely on it, we only give a brief definition here. For more background on smooth manifolds, we refer to \cite[Chapter~5]{prasolov2006elements} and \cite{lee2013smooth}.

Given a connected and open subset $U \subset \manifold$ and a homeomorphism $\varphi\colon U \rightarrow \varphi(U)$ onto an open subset of $\mathbb{R}^d$, the pair $(U,\varphi)$ is called a \emph{chart}. Given two charts $(U_\alpha,\varphi_\alpha)$ and $(U_\beta,\varphi_\beta)$ with $U_\alpha \cap U_\beta \neq \emptyset$, the map $\tau_{\alpha,\beta}\colon \varphi_\alpha(U_\alpha \cap U_\beta) \rightarrow \varphi_\beta(U_\alpha \cap U_\beta)$ defined via $\tau_{\alpha,\beta} = \varphi_\beta \circ \varphi_\alpha^{-1}$ is called a \emph{transition map}.  A \emph{smooth structure} on a topological manifold $\manifold$ with $\partial\manifold = \emptyset$ is a collection $\atlas = \{(U_\alpha,\varphi_\alpha) : \alpha \in A\}$ of charts that satisfies the following three properties.
\begin{enumerate}
	\item The sets $U_\alpha$ cover $\manifold$, that is $\bigcup_{\alpha \in A} U_\alpha = \manifold$.
	\item For any $\alpha, \beta \in A$ with $U_\alpha \cap U_\beta \neq \emptyset$, the transition map $\tau_{\alpha,\beta}$ is smooth.\footnote{See \cite[p.\ 185]{prasolov2006elements} for a discussion of (smooth) maps of manifolds.}
	\item The collection $\atlas$ is maximal in the sense that if $(U,\varphi)$ is a chart and for every $\alpha \in A$ with $U \cap U_\alpha \neq \emptyset$ the transition maps $\varphi \circ \varphi_\alpha^{-1}$ and $\varphi_\alpha \circ \varphi^{-1}$ are smooth, then $(U,\varphi) \in \atlas$.
\end{enumerate}

A topological manifold together with a smooth structure is called a \emph{smooth manifold}. By an appropriate modification of property \mynum{2} above, the definition extends to manifolds with non-empty boundary as well. We refer to \cite[Section 5.1.1]{prasolov2006elements} for details.

\subparagraph*{Triangulations.} Let $\manifold$ be a compact topological $d$-manifold. A simplicial complex $\tri$ whose geometric realization $\|\tri\|$ is homeomorphic to $\manifold$ is called a \emph{triangulation} of $\manifold$.  It follows that $\tri$ is a pure simplicial complex of dimension $d$ (see \Cref{ssec:complexes}). For $d \leq 3$, every topological $d$-manifold admits a triangulation \cite{moise1952affine, rado1925riemannschen}, however, for $d > 3$ this is generally not true (see \cite{manolescu2016lectures} for an overview). Smooth manifolds can nevertheless always be triangulated, irrespective of their dimension, e.g., by work of Whitney \cite[Chapter IV.B]{whitney1957geometric} (cf.\ \Cref{sec:whitney}). 

Given a $d$-dimensional triangulation $\tri$, recall that its \emph{dual graph} $\dual(\tri)$ is the graph with vertices corresponding to the $d$-simplices of $\tri$, and edges to the face gluings, i.e, those $(d-1)$-simplices of $\tri$ that are contained in precisely two $d$-simplices. Note that $\deg(v) \leq d+1$ for any vertex $v$ of $\dual(\tri)$. The proof of the following proposition is left to the reader.

\begin{proposition}
\label{prop:submanifold-induced}
Let $\tri$ be a triangulation of a $d$-manifold $\manifold$ and $\mathscr{U}$ be a collection of $d$-simplices of $\tri$ that define a submanifold of $\manifold$. Then $\dual(\mathscr{U})$ is an induced subgraph of $\dual(\tri)$.
\end{proposition}

\section{Whitney's method}
\label{sec:whitney}

A seminal result of Whitney states that a smooth $d$-dimensional manifold $\manifold$ always admits a smooth embedding into a $2d$-dimensional Euclidean space.

\begin{theorem}[strong Whitney embedding theorem {\cite[Theorem 5]{whitney1944intersections}}, cf.\ {\cite[Theorem 6.19]{lee2013smooth}}]
For $d > 0$, every smooth $d$-manifold admits a smooth embedding into $\mathbb{R}^{2d}$.
\label{thm:strong-whitney}
\end{theorem}

An important consequence of \Cref{thm:strong-whitney} is that smooth manifolds can be triangulated.

\begin{theorem}[triangulation theorem {\cite[Chapter IV.B]{whitney1957geometric}}, cf.\ {\cite[Theorem 1.1]{boissonnat2021triangulating}}] Every compact, smooth $d$-manifold $\manifold$ embedded in some Euclidean space $\mathbb{R}^m$ admits a triangulation.
\label{thm:whitney-trg}
\end{theorem}

Next, we give a very brief and high-level overview of Whitney's method of triangulating smooth manifolds based on \cite[Chapter IV.B]{whitney1957geometric} sufficient for our purposes. To the interested reader we also recommend \cite{boissonnat2021triangulating}, where Whitney's method is recast in a computational setting.

\subparagraph*{Triangulating smooth manifolds.}
Let $\manifold$ be a compact smooth $d$-manifold. By \Cref{thm:strong-whitney} there exists a smooth embedding $\iota\colon\manifold\rightarrow\mathbb{R}^{2d}$. Given such an embedding, we choose a sufficiently fine (with respect to $\iota$) hypercubic honeycomb decomposition of $\mathbb{R}^{2d}$, denoted by $L_0$.
Next, we pass to the first barycentric subdivision $L$ of $L_0$. (Geometrically, $L$ is obtained from $L_0$ by subdividing each $k$-dimensional hypercube of $L_0$ into $(2k)!!$ simplices.) By slightly perturbing the vertices of $L$ we obtain a triangulation $L^\ast$ of $\mathbb{R}^{2d}$, which is combinatorially isomorphic to $L$, but is in \emph{general position} with respect to $\iota(\manifold) \subset \mathbb{R}^{2d}$. Now, by the work of Whitney, $L^\ast$ induces a triangulation $\tri$ of $\manifold$, where, importantly, $\tri$ is a subcomplex of the $d$-skeleton of the barycentric subdivision $(L^\ast)'$ of $L^\ast$. 

See \Cref{fig:whitney} for an illustration of this procedure for $d=1$.

\begin{figure}[thbp]
	\centering
	\begin{minipage}[t]{.27\linewidth}
		\includegraphics{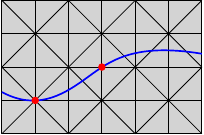}
		\subcaption{The triangulation $L$ and the image $\iota(\manifold)$ (blue). The red points indicate the vertices of $L$ contained by $\iota(\manifold)$.}\label{fig:whitney-base}
	\end{minipage}\hfill%
	\begin{minipage}[t]{.27\linewidth}
		\includegraphics{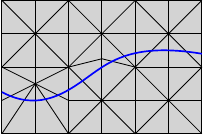}
		\subcaption{The perturbed triangulation $L^\ast$, which in general position with respect to  $\iota(\manifold)$.}\label{fig:whitney-perturb}
	\end{minipage}\hfill%
	\begin{minipage}[t]{.27\linewidth}
		\includegraphics{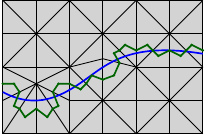}
		\subcaption{The resulting triangulation $\tri$ of $\manifold$ (dark green). $\tri$ is a subcomplex of $(L^\ast)'$.}\label{fig:whitney-approx}
	\end{minipage}
	\caption{Illustration of Whitney's method of triangulating ambient submanifolds ($d=1$).\label{fig:whitney}}
\end{figure}

Similar to \Cref{prop:submanifold-induced}, the next proposition is a direct consequence of the definitions.

\begin{proposition}
\label{prop:whitney-dual-induced}
For any \emph{Whitney triangulation}\footnote{Recall that a triangulation $\tri$ of a compact smooth manifold $\manifold$ is called a \emph{Whitney triangulation}, if $\tri$ is obtained via Whitney's method discussed in \Cref{sec:whitney}.} $\tri$ of a compact smooth $d$-manifold $\manifold$, we have that the dual graph $\dual(\tri)$ is an induced subgraph of $\dual(((\hcubic^{2d,n})'')_d)$.
\end{proposition}

\section{The proof of Theorem \ref{thm:tww-mfd}}
\label{sec:proof}

In this section we prove our main result, i.e., every compact smooth $d$-manifold $\manifold$  has twin-width $\tww(\manifold) \leq d^{O(d)}$.
To streamline the notation, we let $G_{d,n} = \dual(((\hcubic^{2d,n})'')_d)$, i.e., $G_{d,n}$ denotes the dual graph of the $d$-skeleton\footnote{Recall that the dual graph $\dual(\complex)$ of a pure $k$-dimensional complex $\complex$ has vertices corresponding to the $k$-faces of $\complex$ and two vertices are connected if and only if their corresponding $k$-faces share a $(k-1)$-face.} of the second barycentric subdivision of the hybercubic honeycomb $\hcubic^{2d,n}$.
The result is based on the following property of $G_{d,n}$.

\begin{theorem}
\label{thm:tww-Gd}
For the twin-width of the dual graph $G_{d,n} = \dual(((\hcubic^{2d,n})'')_d)$ of the $d$-skeleton of the second barycentric subdivision of the hypercubic honeycomb $\hcubic^{2d,n}$ we have
\[
	\tww(G_{d,n}) \leq d^{O(d)}.
\]
\end{theorem}

Before proving \Cref{thm:tww-Gd}, we show how it implies \Cref{thm:tww-mfd}.

\begin{proof}[Proof of \Cref{thm:tww-mfd}]
Let $\manifold$ be a compact, smooth $d$-dimensional manifold. Consider a~Whitney triangulation $\tri$ of $\manifold$. By \Cref{prop:whitney-dual-induced}, $\dual(\tri)$ is an induced subgraph of $G_{d,n}$ and by \Cref{thm:tww-Gd}, $\tww(G_{d,n}) \leq d^{O(d)}$. Hence, since twin-width is monotone under taking induced subtrigraphs (\Cref{prop:induced}), we obtain $\tww(\dual(\tri))) \leq \tww(G_{d,n}) \leq d^{O(d)}$.
\end{proof}

To complete the proof of \Cref{thm:tww-mfd}, it remains to show \Cref{thm:tww-Gd}.

\begin{proof}[Proof of \Cref{thm:tww-Gd}]
We establish the theorem by exhibiting a~$d^{O(d)}$-contraction sequence $\seq\colon G_{d,n}\leadsto\bullet$.
We will obtain $\seq$ by concatenating two contraction sequences $\seq_1\colon G_{d,n}\leadsto G_{d,n}^\ast$ and $\seq_2\colon G_{d,n}^\ast\leadsto\bullet$, referred to as the \emph{first} and the \emph{second epoch}, where $G_{d,n}^\ast$ is an appropriate subtrigraph of $\red(\griddiag_{n,2d})$, the $2d$-dimensional red $n$-grid with diagonals.
In the following we regard $\hcubic^{2d,n}$ and its subdivisions mainly as abstract complexes, but we will also take advantage of their geometric nature.

\proofsubparagraph{Preparations.} Consider the $(2d+1)$-coloring  $\coloring \colon \hcubic^{2d,n} \rightarrow \{0,\ldots,2d\}$, where we color the cubes of $\hcubic^{2d,n}$ by their dimension, that is, for $c \in \hcubic^{2d,n}$ we set $\coloring(c) = \dim(c)$ (\Cref{fig:h-2-4-colored}).

The coloring $\coloring$ induces a $(2d+1)$-coloring $\coloring'' \colon (\hcubic^{2d,n})'' \rightarrow  \{0,\ldots,2d\}$ of the simplices of the second barycentric subdivision $(\hcubic^{2d,n})''$ as follows. The vertices of the first barycentric subdivision $(\hcubic^{2d,n})'$ are in one-to-one correspondence with the cubes in $\hcubic^{2d,n}$, hence $\coloring$ induces a coloring $\coloring'_0\colon(\hcubic^{2d,n})'(0)\rightarrow \{0,\ldots,2d\}$ via $\coloring'_0(v_c)=\coloring(c)$, where $v_c$ denotes the vertex of $(\hcubic^{2d,n})'$ corresponding to the cube $c \in \hcubic^{2d,n}$ (\Cref{fig:h-2-4-sd1-colored}). Geometrically, $v_c$ is in the barycenter of the cube $c$. When we pass to the second barycentric subdivision, the vertices of $(\hcubic^{2d,n})'$ become vertices of $(\hcubic^{2d,n})''$, thus there is a natural inclusion $\iota\colon(\hcubic^{2d,n})'(0)\rightarrow(\hcubic^{2d,n})''(0)$. Let $\vertices = \image(\iota) \subset (\hcubic^{2d,n})''(0)$ be the image of $(\hcubic^{2d,n})'(0)$ under this inclusion~$\iota$. We color the elements of $\vertices$ identically to $\coloring'_0$, that is, we consider the coloring $\coloring''_\vertices\colon\vertices\rightarrow \{0,\ldots,2d\}$ defined as $\coloring''_\vertices(v)=\coloring'_0(\iota^{-1}(v))$, see \Cref{fig:h-2-4-sd2-V-colored}. Now, pick a simplex $\sigma \in (\hcubic^{2d,n})''$ and note that $\bigcup_{v \in \vertices}\clstar(v) = (\hcubic^{2d,n})''$, i.e., the closed stars of the vertices $v \in \vertices$ cover $(\hcubic^{2d,n})''$. Let $\vertices_\sigma = \{v \in \vertices : \sigma \in \clstar(v)\}$. We now define $\coloring''(\sigma)$ as
\begin{align}
	\coloring''(\sigma) = \min\{\coloring''_\vertices(v) : v \in \vertices_\sigma\}.
\label{eq:def-coloring}
\end{align}
In words, $\coloring''(\sigma)$ is defined as the smallest dimension of any cube $c \in \hcubic^{2d,n}$ such that $\sigma$ belongs to the closed star of the vertex $\iota(v_c)$ in $(\hcubic^{2d,n})''$, where $v_c$ is the vertex of $(\hcubic^{2d,n})'$ corresponding to $c$. We refer to \Cref{fig:cube-corner-colored} for an example.

\begin{figure}[htbp]
	\centering
	\begin{minipage}[t]{.3\linewidth}
		\centering
		\includegraphics[scale=1.1]{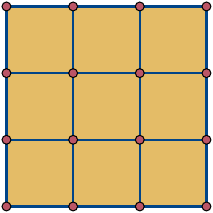}
		\subcaption{The coloring $\coloring$ of the cubes of $\hcubic^{2d,n}$ by their dimension.}\label{fig:h-2-4-colored}
	\end{minipage}\hfill%
	\begin{minipage}[t]{.3\linewidth}
		\centering
		\includegraphics[scale=1.1]{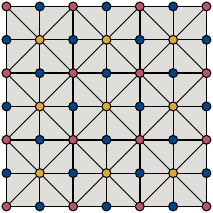}
		\subcaption{The complex $(\hcubic^{2d,n})'$ with its vertices colored by $\coloring'_0$.}\label{fig:h-2-4-sd1-colored}
	\end{minipage}\hfill%
	\begin{minipage}[t]{.3\linewidth}
		\centering
		\includegraphics[scale=1.1]{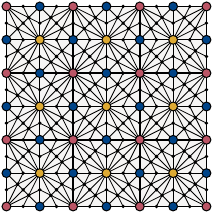}
		\subcaption{The complex $(\hcubic^{2d,n})''$ with its vertices in $\vertices$ colored by $\coloring''_\vertices$.}\label{fig:h-2-4-sd2-V-colored}
	\end{minipage}
	
	\vspace{15pt}
	
	\noindent
	\begin{minipage}[t]{.3\linewidth}
		\centering
		\includegraphics[scale=1.1]{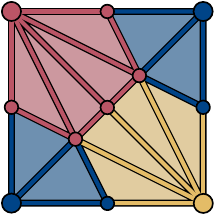}
		\subcaption{The coloring $\coloring''$ near the upper left corner of $(\hcubic^{2d,n})''$. The four larger disks represent vertices that belong to $\vertices$.}\label{fig:cube-corner-colored}
	\end{minipage}\hfill%
	\begin{minipage}[t]{.3\linewidth}
		\centering
		\includegraphics[scale=1.1]{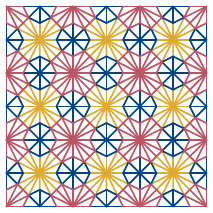}
		\subcaption{The coloring $\coloring''$ restricted to the $d$-simplices of $(\hcubic^{2d,n})''$. The color classes correspond to the families $\mathscr{F}_0$, $\mathscr{F}_1$ and $\mathscr{F}_2$.}\label{fig:h-2-4-sd2-colored}
	\end{minipage}\hfill%
	\begin{minipage}[t]{.3\linewidth}
		\centering
		\includegraphics[scale=1.1]{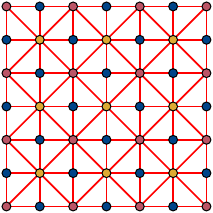}
		\subcaption{\vphantom{$\coloring''$}The trigraph $G_{d,n}^\ast$ obtained from $G_{d,n}$ by contracting each subtrigraph $C_{i,c}$ to a single node. All edges in $G_{d,n}^\ast$ are red.}\label{fig:h-2-4-sd2-contracted}
	\end{minipage}
	\caption{{\sffamily{\bfseries{(a)--(c)}}} Illustrations of the cubical complex $\hcubic^{2d,n}$ for $d=1$ and $n=4$, and of its first and second barycentric subdivisions (which are simplicial complexes) displaying the colorings described in the proof of \Cref{thm:tww-Gd}. {\sffamily{\bfseries{(d)--(f)}}} Three essential steps in the proof of \Cref{thm:tww-Gd}.\label{fig:main-result}}
\end{figure}

\proofsubparagraph{The first epoch.} We start the description of the \emph{first epoch} $\seq_1\colon G_{d,n}\leadsto G_{d,n}^\ast$ by considering the restriction $\coloring''_d\colon(\hcubic^{2d,n})''(d)\rightarrow \{0,\ldots,2d\}$ of the coloring $\coloring''$ to the $d$-simplices of $(\hcubic^{2d,n})''$, see \Cref{fig:h-2-4-sd2-colored}. Note that for each $i \in  \{0,\ldots,2d\}$, the $i$-colored $d$-simplices of $(\hcubic^{2d,n})''$ form a family $\mathscr{F}_i = \{F_{i,c} : c \in \hcubic^{2d,n}(i)\}$ of pairwise-disjoint connected subcomplexes of the \mbox{$d$-skeleton} of~$(\hcubic^{2d,n})''$, where $F_{i,c}$ denotes the subcomplex induced by the $i$-colored $d$-simplices of $(\hcubic^{2d,n})''$ belonging to the closed star of the vertex $\iota(v_c)$, where $v_c$ is the vertex of $(\hcubic^{2d,n})'$ corresponding to the $i$-cube $c$ in $\hcubic^{2d,n}$ (\Cref{fig:h-2-4-sd2-colored}). We also let $\mathscr{F} = \bigcup_{i=0}^{2d} \mathscr{F}_i$.

Since $G_{d,n}$ is defined as the dual graph of the $d$-skeleton of $(\hcubic^{2d,n})''$, the nodes of $G_{d,n}$ are in one-to-one correspondence with the $d$-simplices of $(\hcubic^{2d,n})''$. Let $\gamma\colon (\hcubic^{2d,n})''(d) \rightarrow V(G_{d,n})$ denote the bijection realizing this correspondence. Henceforth, by a slight abuse of notation, we also consider $\coloring''_d$ to be a $(2d+1)$-coloring of $V(G_{d,n})$ via $\coloring''_d(v) = \coloring''_d(\gamma^{-1}(v))$. Let $C_{i,c}$ denote the subtrigraph of $G_{d,n}$ induced by the nodes $\{\gamma(\sigma): \sigma \in F_{i,c}\}$. Note that $\coloring''_d$ assigns the color $i$ to all nodes of $C_{i,c}$. The first epoch $\seq_1\colon G_{d,n}\leadsto G_{d,n}^\ast$ is merely the contraction sequence, where we contract each $C_{i,c}$ (where $0\leq i \leq 2d$ and $c$ is running over $\hcubic^{2d,n}$) to a~single node, in any order, obtaining a trigraph $G_{d,n}^\ast$ (\Cref{fig:h-2-4-sd2-contracted}).

\proofsubparagraph{The second epoch.} $\seq_2\colon G_{d,n}^\ast\leadsto\bullet$ is defined as an optimal contraction sequence of the trigraph $G_{d,n}^\ast$ to a single vertex. Since $G_{d,n}^\ast$ is a subtrigraph of $\red(\griddiag_{n,2d})$, the \emph{$2d$-dimensional red $n$-grid with diagonals}, by \Cref{thm:tww-grid-diag} the width of $\seq_2$ is at most $2 (3^{2d}-1)$.

\proofsubparagraph{Bounding the width of the first epoch.} We now give a rough estimate to show that the first epoch $\seq_1\colon G_{d,n}\leadsto G_{d,n}^\ast$ is a $d^{O(d)}$-contraction sequence. The estimate is based on \Cref{claim:estimate} below, which follows from elementary properties of the hypercubic honeycomb and barycentric subdivisions.

\begin{claim}
For any color $i \in  \{0,\ldots,2d\}$ and cube $c \in \hcubic^{2d,n}$, the subcomplex $F_{i,c}$ of $\hcubic^{2d,n}$ (resp.\ the subtrigraph $C_{i,c}$ of $G_{d,n}$) defined above
\begin{enumerate}
 \item has less than $h_1(d) = 4^{d} ((2d)!)^2 \binom{2d}{d}$ $d$-simplices (resp.\ nodes), and
 \item less than $h_2(d)=9^{d}$  incident subcomplexes $F_{i',c'} \in \mathscr{F}$ (resp.\ adjacent subtrigraphs $C_{i',c'}$).
 \end{enumerate}
\label{claim:estimate}
\end{claim}

Indeed, these two facts imply that by sequentially contracting each $C_{i,c}$ into a single node, the maximum red degree remains bounded by $O(h_1(d)^3 h_2(d))$ throughout the first epoch.

\begin{claimproof}[Proof of \Cref{claim:estimate}]
To bound the number of $d$-simplices in $F_{i,c}$, note that
\begin{itemize}
	\item $F_{i,c}$ is covered by an appropriate translate of a $2d$-dimensional cube of $\hcubic^{2d}$,
	\item the barycentric subdivision of a $2d$-cube contains $2^{2d}(2d)!$ $2d$-simplices (\Cref{claim:bary-cube}),
	\item the barycentric subdivision of a $2d$-simplex contains $(2d)!$ $2d$-simplices (\Cref{claim:bary-simplex}),
	\item a $2d$-simplex has $\binom{2d}{d}$ faces of dimension $d$ (\Cref{claim:simplex-faces}).
\end{itemize}
Multiplying these numbers, we obtain the first part of the claim.

To bound the number of subcomplexes $F_{i',c'} \in \mathscr{F}$ incident to $F_{i,c}$, observe that the incidences between these subcomplexes reflect those between the handles in the canonical handle decomposition of $\hcubic^{2d,n}$ induced by its cubical structure. Thus, the number of subcomplexes in $\mathscr{F}$ incident to $F_{i,c}$ equals $3^{2d}-1$ if $i \in \{0,2d\}$ and $3^i + 3^{2d-i}-2$ otherwise. Both of these numbers are less than $9^d$, so the second part of the claim also holds.
\end{claimproof}

The width of $\seq$ is the maximum of the widths of $\seq_1$ and $\seq_2$, hence $\tww(G_{d,n}) \leq d^{O(d)}$.
\end{proof}

\section{Triangulations with dual graph of arbitrary large twin-width}
\label{sec:large}

In \Cref{sec:proof} we showed that every compact, smooth $d$-manifold admits a triangulation with dual graph of twin-width at most $d^{O(d)}$. In this section we prove complementary results, showing that triangulations with dual graphs of large twin width are abundant. We shed light on this fact in two ways. First, we show that for any fixed dimension $d \geq 3$, the class of $(d+1)$-regular graphs that can be dual graphs of triangulated $d$-manifolds is \emph{not} small. Second, we show that the $d$-dimensional ball admits triangulations with a dual graph of arbitrarily large twin-width, which extends to every piecewise-linear (hence smooth) $d$-manifold. Both of these results rely on counting arguments, and thus are not constructive.

\subsection{The class of dual graphs of triangulations is not small}

Let us fix an integer $d \geq 3$. Following \cite{chapuy2021number}, we let $M_d(n)$ denote the number of colored\footnote{We refer to \cite[Section 2.1]{chapuy2021number} for the precise definitions. Since colored triangulations form a subfamily of uncolored triangulations (those considered in this paper), any lower bound on $M_d(n)$ is automatically a~lower bound on the number of uncolored $d$-dimensional labeled triangulations with $n$ simplices.} triangulations of closed orientable $d$-dimensional manifolds consisting of $n$ $d$-simplices labeled from $1$ to $n$. We assume for convenience that $n$ is even. By \cite[Theorem 1.1]{chapuy2021number} we have\footnote{Here ``$\preceq$'' denotes a comparison where exponential factors are ignored: more precisely, $f(n) \preceq g(n)$ means that there exists some constant $K > 0$, such that for $n$ large enough, we have $f(n) \leq K^n g(n)$.}
\begin{align}
	n! \cdot n^{n/(2d)} \preceq M_d(n).
	\label{eq:colored-trg}
\end{align}
Now, any $(d+1)$-regular graph $G$ with $n$ vertices can be the dual graph of at most $d!^{n(d+1)/2}$ different $d$-dimensional triangulations. Indeed, if $G$ is the dual graph of some $d$-dimensional triangulation, then each of the $n(d+1)/2$ edges of $G$ corresponds to a face gluing, i.e., an identification of two $(d-1)$-dimensional simplices via a simplicial isomorphism, of which there are $d!$ many.\footnote{This is because a~simplicial isomorphism between two $(d-1)$-simplices $\sigma$ and $\tau$ is determined by a~perfect matching between the $d$ vertices of $\sigma$ and the $d$ vertices of~$\tau$.} Hence by \eqref{eq:colored-trg}, for the number $\Gamma_{d}(n)$ of $(d+1)$-regular graphs on $n$ labeled vertices that are dual graphs of some $d$-dimensional triangulations, we have
\begin{align}
	\frac{n! \cdot n^{n/(2d)}}{d!^{n(d+1)/2}} \preceq \Gamma_d(n).
	\label{eq:number-dual}
\end{align}
As the left-hand side of \eqref{eq:number-dual} grows super-exponentially in $n$, the next theorem directly follows.

\begin{theorem}
For every $d \geq 3$, the class of $(d+1)$-regular graphs that are dual graphs of triangulations of $d$-manifolds is not small. In particular, there are graphs with arbitrarily large twin-width in this class.
\label{thm:trg-not-small}
\end{theorem}

\subsection{\texorpdfstring{Complicated triangulations of the $\boldsymbol{d}$-dimensional ball}{Complicated triangulations of the d-dimensional ball}}

In this section we show \Cref{thm:tww-large}, according to which every piecewise-linear (PL) $d$-manifold ($d \geq 3$) admits triangulations with a dual graph of arbitrary large twin-width. To show the existence of such triangulations for every PL-manifold, we rely on the monotonicity of twin-width with respect to taking induced subtrigraphs (\Cref{prop:induced}), the fact that the class of $d$-subdivisions of $(d+1)$-regular graphs is \emph{not} small (\Cref{prop:k-reg-s-sub}) together with the following classical result from the theory of PL-manifolds.

\begin{theorem}[{\cite[Corollary 1]{armstrong1967extending}}]
\label{thm:extending-trg}
Any triangulation of the boundary of a compact piecewise-linear (PL) manifold can be extended to a triangulation of the whole manifold.
\end{theorem}

We first show that already the $d$-dimensional ball $\ball^d = \{x \in \mathbb{R}^d : \|x\| \leq 1\}$ admits triangulations with dual graph of arbitrary large twin-width. More precisely, we prove:

\begin{theorem}
\label{thm:tww-large-ball}
For every $\integer \in \mathbb{N}$ there is a triangulation $\tri_\integer$ of the $d$-dimensional ball~$\ball^d$, such that $\tww(\dual(\tri_\integer)) \geq \integer$ and $\partial\tri_\integer = \partial\Delta^d$, the boundary of the standard $d$-simplex~$\Delta^d$.
\end{theorem}

\begin{proof}
Let $G_\integer$ be a $d$-subdivision of a $(d+1)$-regular graph $G$ such that $\tww(G_\integer) \geq \integer$. The existence of such a graph is guaranteed by \Cref{prop:k-reg-s-sub}. Let $\altmanifold$ be a $d$-manifold homeomorphic to a closed regular neighborhood of a straight-line embedding of $G$ in $\mathbb{R}^d$. Informally, $\altmanifold$ can be seen as a $d$-dimensional thickening of the graph $G$.

Construct an abstract triangulation of $\altmanifold$ as follows. Take a $d$-simplex $\sigma_v$ for each node $v$ of~$G$, and fix a one-to-one correspondence between the $d+1$ facets of $\sigma_v$ and the $d+1$ arcs incident to $v$ in $G$. For every arc $\{u,v\} \in E(G)$, take a simplicial $d$-prism $P_{\{u,v\}} = \sigma_{\{u,v\}} \times [0,1]$, where $\sigma_{\{u,v\}}$ is a $(d-1)$-simplex, and attach $P_{\{u,v\}}$ to the simplices $\sigma_u$ and $\sigma_v$ by identifying $\sigma_{\{u,v\}} \times \{0\}$ (resp.\ $\sigma_{\{u,v\}} \times \{1\}$) with the facet of $\sigma_u$ (resp.\ $\sigma_v$) that corresponds to the arc $\{u,v\}$. Now triangulate each prism $P_{\{u,v\}}$ with a minimal triangulation consisting of $d$ $d$-simplices stacked onto each other, see \Cref{ex:trg-prism} below.

Let $\tri_\altmanifold$ denote the resulting triangulation of $\altmanifold$.

\begin{claim}
\label{claim:dual}
For the dual graph of the triangulation $\tri_\altmanifold$ we have $\dual(\tri_\altmanifold)=G_\integer$.
\end{claim}

\begin{claimproof}
The claim follows immediately from the facts that the dual graph of the constructed triangulation of the simplicial $d$-prism is merely a path of length $d$, and $G_\integer$ is the $d$-subdivision of the $(d+1)$-regular graph $G$ on which the triangulation $\tri_\altmanifold$ is modeled.
\end{claimproof}

Pick a sufficiently large $\ell \in \mathbb{N}$ such that the $\ell$\textsuperscript{th} iterated barycentric subdivision $\tri_\altmanifold^{(\ell)}$ of $\tri_\altmanifold$ embeds linearly in $\mathbb{R}^d$ and fix a simplex-wise linear embedding $\embedding\colon\geomrel{\tri_\altmanifold^{(\ell)}} \rightarrow \mathbb{R}^d$. Consider a large geometric $d$-simplex $\bigsimplex \subset \mathbb{R}^d$ that contains the image $\image(\embedding)$ of $\embedding$ in its interior. Now $\bigsimplex_\circ = \bigsimplex \setminus \interior(\image(\embedding))$ is a PL manifold\footnote{It is folklore that every codimension zero submanifold of a Euclidean space is a PL manifold, see, e.g., \cite[p.\ 118]{levitt1987intrinsic} or \cite[Remark 1.1.10]{waldhausen2013spaces}.} with triangulated boundary, so by \Cref{thm:extending-trg} this boundary triangulation can be extended to a triangulation $\tri_{\bigsimplex_\circ}$ of the entire manifold $\bigsimplex_\circ$. The boundary of $\tri_{\bigsimplex_\circ}$ has two connected components: $\partial_1\tri_{\bigsimplex_\circ} \cong \partial (\tri_\altmanifold^{(\ell)})$ and $\partial_2\tri_{\bigsimplex_\circ} \cong \partial\Delta^d$.

Let $\altaltmanifold$ be a $d$-manifold homeomorphic to $\partial\tri_\altmanifold \times [0,1]$. Take a triangulation $\tri_\altaltmanifold$ of $\altaltmanifold$, such that for the two boundary components we have $\partial_1\tri_\altaltmanifold \cong \partial\tri_\altmanifold$ and $\partial_2\tri_\altaltmanifold \cong \partial(\tri_\altmanifold^{(\ell)})$. One way to construct such a triangulation is as follows. First, consider the decomposition $\prisms_\altaltmanifold$ of $\altaltmanifold$ into simplicial $d$-prisms induced by the product structure $\partial\tri_\altmanifold \times [0,1]$. That is, $\prisms_\altaltmanifold$ consists of $d$-prisms $\prism_\sigma \cong \sigma \times [0,1]$, one for each $(d-1)$-simplex $\sigma$ of $\partial\tri_\altmanifold$, glued together along their vertical boundary prisms the same way as the simplices of $\partial\tri_\altmanifold$. The boundary $\partial\prisms_\altaltmanifold$ of $\prisms_\altaltmanifold$ has two connected components: $\partial_1\prisms_\altaltmanifold$ and $\partial_2\prisms_\altaltmanifold$, each combinatorially isomorphic to $\partial\tri_\altmanifold$. Next, pass to the $\ell$\textsuperscript{th} iterated barycentric subdivision of $\partial_2\prisms_\altaltmanifold$. This operation turns each prism $\prism_\sigma$ of $\prisms_\altaltmanifold$ into a polyhedral cell $R_\sigma$. These cells form a polyhedral decomposition $\polydecomp_\altaltmanifold$ of $\altaltmanifold$, where $\partial_1\polydecomp_\altaltmanifold \cong \partial\tri_\altmanifold$ and $\partial_2\polydecomp_\altaltmanifold \cong (\partial\tri_\altmanifold)^{(\ell)} = \partial(\tri_\altmanifold^{(\ell)})$. Triangulate $R_\sigma$ as follows. Consider an order $c_1 \prec c_2 \prec \cdots$ of the vertical cells\footnote{These are precisely those cells that are not contained in the boundary of $\polydecomp_\altaltmanifold$.} of $R_\sigma$, where $c_i \prec c_j$ implies $\dim(c_i) \leq \dim(c_j)$. Place a new vertex $v_i$ in the barycenter of $c_i$ and, iterating over the vertical cells in the above order, triangulate $c_i$ by coning from $v_i$ over its (already triangulated) boundary $\partial c_i$. It is clear that the resulting triangulation of $R_\sigma$ is symmetric with respect to the symmetries of its base simplex $\sigma$. Applying this procedure for each polyhedral cell $R_\sigma$ of $\polydecomp_\altaltmanifold$ yields a~triangulation $\tri_\altaltmanifold$ of $\altaltmanifold$ with the desired properties.

Now the triangulation $\tri_\integer$ of $\ball^d$ is obtained by gluing together $\tri_\altmanifold$, $\tri_\altaltmanifold$ and $\tri_{\bigsimplex_\circ}$ via the identity maps along the isomorphic boundary-pairs $\partial\tri_\altmanifold \cong \partial_1\tri_\altaltmanifold$ and $\partial_2\tri_\altaltmanifold \cong \partial_1\tri_{\bigsimplex_\circ}$.

To conclude, note that the $d$-simplicies of $\tri_\altmanifold$ triangulate a submanifold of $\|\tri_\integer\|$, hence by \Cref{prop:submanifold-induced} the graph $\dual(\tri_\altmanifold)$ is an induced subgraph of $\dual(\tri_\integer)$. By \Cref{claim:dual}, $\dual(\tri_\altmanifold) = G_\integer$ and by the initial assumption $\tww(G_\integer) \geq \integer$, hence by \Cref{prop:induced}, $\dual(\tri_\integer) \geq \integer$ as well.
\end{proof}

\begin{proof}[Proof of \Cref{thm:tww-large}]
Let $\manifold$ be an arbitrary PL-manifold possibly with non-empty boundary and $\tri_\circ$ be a simplicial triangulation of $\manifold$. Consider a triangulation $\tri_\integer$ of the $d$-ball with $\tww(\dual(\tri_\integer)) \geq \integer$ as in \Cref{thm:tww-large-ball}. Let $\Delta$ be a $d$-simplex of $\tri_\circ$ that is disjoint from $\partial\manifold$ (if $\tri_\circ$ does not contain such a simplex, just replace $\tri_\circ$ with its second barycentric subdivision). Since $\tri_\circ$ is simplicial, $\Delta$ is embedded in $\tri_\circ$ and is topologically a $d$-ball. Now replace $\Delta$ with $\tri_\integer$ by first removing $\Delta$ from $\tri_\circ$, thereby creating a boundary component isomorphic to $\partial\Delta$, then gluing $\partial\tri_\integer$ to this new boundary component via a simplicial isomorphism. (Note that this is possible since $\partial\tri_\integer \cong \partial\Delta$.) Let $\tri$ denote the resulting triangulation of $\manifold$. By \Cref{prop:induced,prop:submanifold-induced}, it follows that $\tww(\dual(\tri)) \geq \tww(\dual(\tri_\integer)) \geq \integer$.
\end{proof}

\begin{example}[triangulating simplicial $d$-prisms]\label{ex:trg-prism}
Let $\sigma$ be a $(d-1)$-simplex with vertices $\{v_1,\ldots,v_d\}$, and let $\prism_\sigma = \sigma \times [0,1]$ be its associated simplicial $d$-prism. Note that for the vertex set of $\prism_\sigma$ we have $\prism_\sigma(0) = \{(v_1,0),\ldots,(v_d,0),(v_1,1),\ldots,(v_d,1)\}$. A triangulation of $\prism_\sigma$ with $d$-simplices $\{\overline{\sigma}_1,\ldots,\overline{\sigma}_d\}$ can be obtained as follows. We define $\overline{\sigma}_i$ iteratively, through their vertex sets. First, set $\overline{\sigma}_1 = \{ (v_1,0),\ldots,(v_d,0),(v_1,1)\}$. Next, for $2 \leq i \leq d$ the vertex set of $\overline{\sigma}_i$ is simply obtained from that of $\overline{\sigma}_{i-1}$ by replacing $(v_i,0)$ with $(v_i,1)$.
\end{example}

\begin{figure}[htbp]
	\centering
	
	\bigskip
	\begin{overpic}{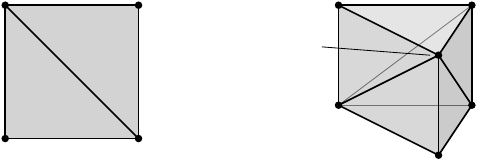}
		\put (-11,-1){\small{$(v_1,0)$}}
		\put (30.75,-1){\small{$(v_2,0)$}}
		\put (-11,36){\small{$(v_1,1)$}}
		\put (30.75,36){\small{$(v_2,1)$}}
		\put (7.5,11){\small{$\overline{\sigma}_1$}}
		\put (19.5,23){\small{$\overline{\sigma}_2$}}
		\put (86.75,-4.25){\small{$(v_1,0)$}}
		\put (101,8){\small{$(v_2,0)$}}
		\put (58,8){\small{$(v_3,0)$}}
		\put (101,36){\small{$(v_2,1)$}}
		\put (58,36){\small{$(v_3,1)$}}
		\put (56,23){\small{$(v_1,1)$}}
	\end{overpic}
	\bigskip

	\caption{The considered triangulations of the simplicial $d$-prism for $d=2$ and $d=3$.}
	\label{fig:prism}
\end{figure}
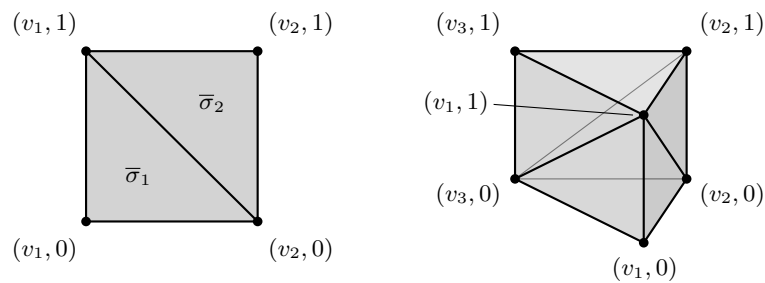


\bibliography{references}

\begin{thebibliography}{10}

\bibitem{armstrong1967extending}
M.~A. Armstrong.
\newblock Extending triangulations.
\newblock {\em Proc. Amer. Math. Soc.}, 18:701--704, 1967.
\newblock \href {https://doi.org/10.2307/2035442} {\path{doi:10.2307/2035442}}.

\bibitem{bagchi2016efficient}
B.~Bagchi, B.~A. Burton, B.~Datta, N.~Singh, and J.~Spreer.
\newblock Efficient algorithms to decide tightness.
\newblock In {\em 32nd {I}nt. {S}ymp. {C}omput. {G}eom. ({SoCG} 2016)},
  volume~51 of {\em LIPIcs. Leibniz Int. Proc. Inf.}, pages 12:1--12:15.
  Schloss Dagstuhl--Leibniz-Zent. Inf., 2016.
\newblock \href {https://doi.org/10.4230/LIPIcs.SoCG.2016.12}
  {\path{doi:10.4230/LIPIcs.SoCG.2016.12}}.

\bibitem{barnatan2007fast}
D.~Bar-Natan.
\newblock Fast {K}hovanov homology computations.
\newblock {\em J. Knot Theory Ramifications}, 16(3):243--255, 2007.
\newblock \href {https://doi.org/10.1142/S0218216507005294}
  {\path{doi:10.1142/S0218216507005294}}.

\bibitem{berge2023approximating}
P.~Berg\'{e}, \'{E}. Bonnet, H.~D\'{e}pr\'{e}s, and R.~Watrigant.
\newblock Approximating highly inapproximable problems on graphs of bounded
  twin-width.
\newblock In {\em 40th {I}nt. {S}ymp. {T}heor. {A}spects {C}omput. {S}ci.
  (STACS 2023)}, volume 254 of {\em LIPIcs. Leibniz Int. Proc. Inform.}, pages
  10:1--10:15. Schloss Dagstuhl. Leibniz-Zent. Inform., 2023.
\newblock \href {https://doi.org/10.4230/lipics.stacs.2023.10}
  {\path{doi:10.4230/lipics.stacs.2023.10}}.

\bibitem{boissonnat2021triangulating}
J.-D. Boissonnat, S.~Kachanovich, and M.~Wintraecken.
\newblock Triangulating submanifolds: an elementary and quantified version of
  {W}hitney's method.
\newblock {\em Discrete Comput. Geom.}, 66(1):386--434, 2021.
\newblock \href {https://doi.org/10.1007/s00454-020-00250-8}
  {\path{doi:10.1007/s00454-020-00250-8}}.

\bibitem{bonnet2024habil}
{\'{E}}.~Bonnet.
\newblock {\itshape{Twin-width and contraction sequences}}.
\newblock Habilitation thesis, ENS de Lyon, April 2024.
\newblock URL: \url{https://perso.ens-lyon.fr/edouard.bonnet/text/hdr.pdf}.

\bibitem{bonnet2022twin-width8}
\'{E}. Bonnet, D.~Chakraborty, E.~J. Kim, N.~K\"{o}hler, R.~Lopes, and
  S.~Thomass\'{e}.
\newblock Twin-width {VIII}: delineation and win-wins.
\newblock In {\em 17th {I}nt. {S}ymp. {P}arametr. {E}xact {C}omput. (IPEC
  2022)}, volume 249 of {\em LIPIcs. Leibniz Int. Proc. Inform.}, pages
  9:1--9:18. Schloss Dagstuhl. Leibniz-Zent. Inform., 2022.
\newblock \href {https://doi.org/10.4230/lipics.ipec.2022.9}
  {\path{doi:10.4230/lipics.ipec.2022.9}}.

\bibitem{bonnet2021twin-width-3}
\'{E}. Bonnet, C.~Geniet, E.~J. Kim, S.~Thomass\'{e}, and R.~Watrigant.
\newblock Twin-width {III}: max independent set, min dominating set, and
  coloring.
\newblock In {\em 48th {I}nt. {C}olloq. {A}utom. {L}ang. {P}rog. (ICALP 2021)},
  volume 198 of {\em LIPIcs. Leibniz Int. Proc. Inform.}, pages 35:1--35:20.
  Schloss Dagstuhl. Leibniz-Zent. Inform., 2021.
\newblock \href {https://doi.org/10.4230/LIPIcs.ICALP.2021.35}
  {\path{doi:10.4230/LIPIcs.ICALP.2021.35}}.

\bibitem{bonnet2022twin-width-2}
\'{E}. Bonnet, C.~Geniet, {E.\ J.} Kim, S.~Thomass\'{e}, and R.~Watrigant.
\newblock Twin-width {II}: small classes.
\newblock {\em Comb. Theory}, 2(2):Paper No. 10, 42, 2022.
\newblock \href {https://doi.org/10.5070/C62257876}
  {\path{doi:10.5070/C62257876}}.

\bibitem{bonnet2022twin-width-1}
\'{E}. Bonnet, E.~J. Kim, S.~Thomass\'{e}, and R.~Watrigant.
\newblock Twin-width {I}: {T}ractable {FO} model checking.
\newblock {\em J. ACM}, 69(1):Art. 3, 46, 2022.
\newblock \href {https://doi.org/10.1145/3486655} {\path{doi:10.1145/3486655}}.

\bibitem{burton2013regina}
B.~A. Burton.
\newblock Computational topology with {R}egina: algorithms, heuristics and
  implementations.
\newblock In {\em Geometry and Topology Down Under}, volume 597 of {\em
  Contemp. Math.}, pages 195--224. Am. Math. Soc., Providence, RI, 2013.
\newblock \href {https://doi.org/10.1090/conm/597/11877}
  {\path{doi:10.1090/conm/597/11877}}.

\bibitem{burton2018homfly}
B.~A. Burton.
\newblock The {HOMFLY}-{PT} polynomial is fixed-parameter tractable.
\newblock In {\em 34th {I}nt. {S}ymp. {C}omput. {G}eom. ({SoCG} 2018)},
  volume~99 of {\em LIPIcs. Leibniz Int. Proc. Inform.}, pages 18:1--18:14.
  Schloss Dagstuhl--Leibniz-Zent. Inf., 2018.
\newblock \href {https://doi.org/10.4230/LIPIcs.SoCG.2018.18}
  {\path{doi:10.4230/LIPIcs.SoCG.2018.18}}.

\bibitem{regina}
B.~A. Burton, R.~Budney, W.~Pettersson, et~al.
\newblock Regina: Software for low-dimensional topology, 1999--2023.
\newblock Version 7.3.
\newblock URL: \url{https://regina-normal.github.io}.

\bibitem{burton2017courcelle}
B.~A. Burton and R.~G. Downey.
\newblock Courcelle's theorem for triangulations.
\newblock {\em J. Comb. Theory, Ser. {A}}, 146:264--294, 2017.
\newblock \href {https://doi.org/10.1016/j.jcta.2016.10.001}
  {\path{doi:10.1016/j.jcta.2016.10.001}}.

\bibitem{MFOReports}
B.~A. Burton, H.~Edelsbrunner, J.~Erickson, and S.~Tillmann, editors.
\newblock {\em Computational Geometric and Algebraic Topology}, volume~12 of
  {\em Oberwolfach Rep.} EMS Publ. House, 2015.
\newblock \href {https://doi.org/10.4171/OWR/2015/45}
  {\path{doi:10.4171/OWR/2015/45}}.

\bibitem{burton2016parameterized}
B.~A. Burton, T.~Lewiner, J.~Paix{\~{a}}o, and J.~Spreer.
\newblock Parameterized complexity of discrete {M}orse theory.
\newblock {\em {ACM} Trans. Math. Softw.}, 42(1):6:1--6:24, 2016.
\newblock \href {https://doi.org/10.1145/2738034} {\path{doi:10.1145/2738034}}.

\bibitem{burton2018algorithms}
B.~A. Burton, C.~Maria, and J.~Spreer.
\newblock Algorithms and complexity for {Turaev--Viro} invariants.
\newblock {\em J. Appl. Comput. Topol.}, 2(1--2):33--53, 2018.
\newblock \href {https://doi.org/10.1007/s41468-018-0016-2}
  {\path{doi:10.1007/s41468-018-0016-2}}.

\bibitem{pettersson2014fixed}
B.~A. Burton and W.~Pettersson.
\newblock Fixed parameter tractable algorithms in combinatorial topology.
\newblock In {\em Proc. 20th Int. Conf. Comput. Comb. ({COCOON} 2014)}, volume
  8591 of {\em Lect. Notes Comput. Sci.}, pages 300--311. Springer, 2014.
\newblock \href {https://doi.org/10.1007/978-3-319-08783-2_26}
  {\path{doi:10.1007/978-3-319-08783-2_26}}.

\bibitem{burton2013complexity}
B.~A. Burton and J.~Spreer.
\newblock The complexity of detecting taut angle structures on triangulations.
\newblock In {\em Proc. 24th Annu. {ACM-SIAM} Symp. Discrete Algorithms ({SODA}
  2013)}, pages 168--183, 2013.
\newblock \href {https://doi.org/10.1137/1.9781611973105.13}
  {\path{doi:10.1137/1.9781611973105.13}}.

\bibitem{chapuy2021number}
G.~Chapuy and G.~Perarnau.
\newblock On the number of coloured triangulations of {$d$}-manifolds.
\newblock {\em Discrete Comput. Geom.}, 65(3):601--617, 2021.
\newblock \href {https://doi.org/10.1007/s00454-020-00189-w}
  {\path{doi:10.1007/s00454-020-00189-w}}.

\bibitem{mesmay2019treewidth}
A.~{de Mesmay}, J.~Purcell, S.~Schleimer, and E.~Sedgwick.
\newblock On the tree-width of knot diagrams.
\newblock {\em J. Comput. Geom.}, 10(1):164--180, 2019.
\newblock \href {https://doi.org/10.20382/jocg.v10i1a6}
  {\path{doi:10.20382/jocg.v10i1a6}}.

\bibitem{dvorak2010small}
Z.~Dvo\v{r}\'{a}k and S.~Norine.
\newblock Small graph classes and bounded expansion.
\newblock {\em J. Combin. Theory Ser. B}, 100(2):171--175, 2010.
\newblock \href {https://doi.org/10.1016/j.jctb.2009.06.001}
  {\path{doi:10.1016/j.jctb.2009.06.001}}.

\bibitem{gajarsky2022twin-width}
J.~Gajarsk\'{y}, M.~Pilipczuk, W.~Przybyszewski, and Sz. Toru\'{n}czyk.
\newblock Twin-width and types.
\newblock In {\em 49th {I}nt. {C}onf. {A}utom. {L}ang. {P}rog. (ICALP 2022)},
  volume 229 of {\em LIPIcs. Leibniz Int. Proc. Inform.}, pages 123:1--123:21.
  Schloss Dagstuhl. Leibniz-Zent. Inform., 2022.
\newblock \href {https://doi.org/10.4230/lipics.icalp.2022.123}
  {\path{doi:10.4230/lipics.icalp.2022.123}}.

\bibitem{hlineny2023tww-planar}
P.~Hlin\v{e}n\'{y} and J.~Jedelsk\'{y}.
\newblock Twin-width of planar graphs is at most 8, and at most 6 when
  bipartite planar.
\newblock In {\em 50th {I}nt. {C}olloq. {A}utom. {L}ang. {P}rog. (ICALP 2023)},
  volume 261 of {\em LIPIcs. Leibniz Int. Proc. Inform.}, pages Art. No. 75,
  18. Schloss Dagstuhl. Leibniz-Zent. Inform., 2023.
\newblock \href {https://doi.org/10.4230/lipics.icalp.2023.75}
  {\path{doi:10.4230/lipics.icalp.2023.75}}.

\bibitem{huszar2020combinatorial}
K.~Husz\'ar.
\newblock {\em Combinatorial width parameters for {$3$}-dimensonal manifolds}.
\newblock PhD thesis, IST Austria, June 2020.
\newblock \href {https://doi.org/10.15479/AT:ISTA:8032}
  {\path{doi:10.15479/AT:ISTA:8032}}.

\bibitem{huszar2022pathwidth}
K.~Husz\'ar.
\newblock On the pathwidth of hyperbolic 3-manifolds.
\newblock {\em Comput. Geom. Topol.}, 1(1):1--19, 2022.
\newblock \href {https://doi.org/10.57717/cgt.v1i1.4}
  {\path{doi:10.57717/cgt.v1i1.4}}.

\bibitem{huszar2019manifold}
K.~Husz\'ar and J.~Spreer.
\newblock 3-{M}anifold triangulations with small treewidth.
\newblock In {\em 35th {I}nt. {S}ymp. {C}omput. {G}eom. ({SoCG} 2019)}, volume
  129 of {\em LIPIcs. Leibniz Int. Proc. Inf.}, pages 44:1--44:20. Schloss
  Dagstuhl--Leibniz-Zent. Inf., 2019.
\newblock \href {https://doi.org/10.4230/LIPIcs.SoCG.2019.44}
  {\path{doi:10.4230/LIPIcs.SoCG.2019.44}}.

\bibitem{huszar2023width}
K.~Husz{\'{a}}r and J.~Spreer.
\newblock On the width of complicated {JSJ} decompositions.
\newblock In {\em 39th {I}nt. {S}ymp. {C}omput. {G}eom. ({SoCG} 2023)}, volume
  258 of {\em LIPIcs. Leibniz Int. Proc. Inf.}, pages 42:1--42:18. Schloss
  Dagstuhl--Leibniz-Zent. Inf., 2023.
\newblock \href {https://doi.org/10.4230/LIPIcs.SoCG.2023.42}
  {\path{doi:10.4230/LIPIcs.SoCG.2023.42}}.

\bibitem{huszar2019treewidth}
K.~Husz\'ar, J.~Spreer, and U.~Wagner.
\newblock On the treewidth of triangulated 3-manifolds.
\newblock {\em J. Comput. Geom.}, 10(2):70--98, 2019.
\newblock \href {https://doi.org/10.20382/jogc.v10i2a5}
  {\path{doi:10.20382/jogc.v10i2a5}}.

\bibitem{kral2023planar}
D.~Kr{\'{a}}{\v{l}} and A.~Lamaison.
\newblock Planar graph with twin-width seven.
\newblock {\em Eur. J. Comb.}, page 103749, 2023.
\newblock \href {https://doi.org/10.1016/j.ejc.2023.103749}
  {\path{doi:10.1016/j.ejc.2023.103749}}.

\bibitem{kral2023tww-surfaces}
D.~Kr{\'{a}}{\v{l}}, K.~Pek{\'{a}}rkov{\'{a}}, and K.~Storgel.
\newblock Twin-width of graphs on surfaces, 2023.
\newblock \href {https://arxiv.org/abs/2307.05811} {\path{arXiv:2307.05811}}.

\bibitem{lee2013smooth}
J.~M. Lee.
\newblock {\em Introduction to smooth manifolds}, volume 218 of {\em Grad.
  Texts Math.}
\newblock Springer, New York, second edition, 2013.
\newblock \href {https://doi.org/10.1007/978-1-4419-9982-5}
  {\path{doi:10.1007/978-1-4419-9982-5}}.

\bibitem{levitt1987intrinsic}
N.~Levitt and A.~Ranicki.
\newblock Intrinsic transversality structures.
\newblock {\em Pacific J. Math.}, 129(1):85--144, 1987.
\newblock \href {https://doi.org/10.2140/pjm.1987.129.85}
  {\path{doi:10.2140/pjm.1987.129.85}}.

\bibitem{lunel2023structural}
C.~Lunel and A.~de~Mesmay.
\newblock {A Structural Approach to Tree Decompositions of Knots and Spatial
  Graphs}.
\newblock In {\em 39th {I}nt. {S}ymp. {C}omput. {G}eom. ({SoCG} 2023)}, volume
  258 of {\em LIPIcs. Leibniz Int. Proc. Inf.}, pages 50:1--50:16. Schloss
  Dagstuhl--Leibniz-Zent. Inf., 2023.
\newblock \href {https://doi.org/10.4230/LIPIcs.SoCG.2023.50}
  {\path{doi:10.4230/LIPIcs.SoCG.2023.50}}.

\bibitem{makowsky2005coloured}
J.~A. Makowsky.
\newblock Coloured {T}utte polynomials and {K}auffman brackets for graphs of
  bounded tree width.
\newblock {\em Discrete Appl. Math.}, 145(2):276--290, 2005.
\newblock \href {https://doi.org/10.1016/j.dam.2004.01.016}
  {\path{doi:10.1016/j.dam.2004.01.016}}.

\bibitem{makowsky2003parameterized}
J.~A. Makowsky and J.~P. Mari\~{n}o.
\newblock The parametrized complexity of knot polynomials.
\newblock {\em J. Comput. Syst. Sci.}, 67(4):742--756, 2003.
\newblock Special issue on parameterized computation and complexity.
\newblock \href {https://doi.org/10.1016/S0022-0000(03)00080-1}
  {\path{doi:10.1016/S0022-0000(03)00080-1}}.

\bibitem{manolescu2016lectures}
C.~Manolescu.
\newblock Lectures on the triangulation conjecture.
\newblock In {\em Proc. 22nd {G}\"{o}kova {G}eom.-{T}opol. {C}onf. (GGT 2015)},
  pages 1--38. Int. Press Boston, 2016.
\newblock URL: \url{https://gokovagt.org/proceedings/2015/manolescu.html}.

\bibitem{maria2021parametrized}
C.~Maria.
\newblock Parameterized complexity of quantum knot invariants.
\newblock In {\em 37th {I}nt. {S}ymp. {C}omput. {G}eom. ({SoCG} 2021)}, volume
  189 of {\em LIPIcs. Leibniz Int. Proc. Inf.}, pages 53:1--53:15. Schloss
  Dagstuhl--Leibniz-Zent. Inf., 2021.
\newblock \href {https://doi.org/10.4230/LIPIcs.SoCG.2021.53}
  {\path{doi:10.4230/LIPIcs.SoCG.2021.53}}.

\bibitem{maria2019treewidth}
C.~Maria and J.~Purcell.
\newblock Treewidth, crushing and hyperbolic volume.
\newblock {\em Algebr. Geom. Topol.}, 19(5):2625--2652, 2019.
\newblock \href {https://doi.org/10.2140/agt.2019.19.2625}
  {\path{doi:10.2140/agt.2019.19.2625}}.

\bibitem{moise1952affine}
E.~E. Moise.
\newblock Affine structures in {$3$}-manifolds. {V}. {T}he triangulation
  theorem and {H}auptvermutung.
\newblock {\em Ann. Math. (2)}, 56:96--114, 1952.
\newblock \href {https://doi.org/10.2307/1969769} {\path{doi:10.2307/1969769}}.

\bibitem{noy2015graphs}
M.~Noy.
\newblock Graphs.
\newblock In {\em Handbook of Enumerative Combinatorics}, Discrete Math. Appl.,
  pages 397--436. Chapman \& Hall / CRC, 2015.
\newblock \href {https://doi.org/10.1201/b18255-12}
  {\path{doi:10.1201/b18255-12}}.

\bibitem{prasolov2006elements}
V.~V. Prasolov.
\newblock {\em Elements of combinatorial and differential topology}, volume~74
  of {\em Grad. Stud. Math.}
\newblock Am. Math. Soc., Providence, RI, 2006.
\newblock Translated from the 2004 Russian original by Olga Sipacheva.
\newblock \href {https://doi.org/10.1090/gsm/074} {\path{doi:10.1090/gsm/074}}.

\bibitem{rado1925riemannschen}
T.~Rad\'o.
\newblock {\"U}ber den {B}egriff der {R}iemannschen {F}l\"ache.
\newblock {\em Acta Sci. Math.}, 2(2):101--121, 1925.

\bibitem{waldhausen2013spaces}
F.~Waldhausen, B.~Jahren, and J.~Rognes.
\newblock {\em Spaces of {PL} manifolds and categories of simple maps}, volume
  186 of {\em Annals Math. Stud.}
\newblock Princeton Univ.\ Press, Princeton, NJ, 2013.
\newblock \href {https://doi.org/10.1515/9781400846528}
  {\path{doi:10.1515/9781400846528}}.

\bibitem{whitney1944intersections}
H.~Whitney.
\newblock The self-intersections of a smooth {$n$}-manifold in {$2n$}-space.
\newblock {\em Ann. of Math. (2)}, 45:220--246, 1944.
\newblock \href {https://doi.org/10.2307/1969265} {\path{doi:10.2307/1969265}}.

\bibitem{whitney1957geometric}
H.~Whitney.
\newblock {\em Geometric Integration Theory}.
\newblock Princeton Univ. Press, Princeton, NJ, 1957.
\newblock \href {https://doi.org/10.1515/9781400877577}
  {\path{doi:10.1515/9781400877577}}.

\end{thebibliography}

\end{document}